\documentclass[12pt,reqno,twoside,letterpaper]{amsart}

\usepackage{hyperref}

\makeatletter
\renewenvironment{proof}[1][\proofname]{\par
  \pushQED{\qed}
  \normalfont \partopsep=\z@skip \topsep=\z@skip
  \trivlist
  \item[\hskip\labelsep
    \itshape
  #1\@addpunct{.}]\ignorespaces
}{
  \popQED\endtrivlist\@endpefalse\vspace{0.3em}
}
\makeatother

\DeclareMathOperator{\sgn}{sgn}
\DeclareMathOperator{\tr}{tr}

\newtheorem{corollary}{Corollary}
\newtheorem{lemma}{Lemma}
\newtheorem{proposition}{Proposition}
\newtheorem{question}{Question}
\newtheorem{theorem}{Theorem}

\newcommand{\M}{\mathbb{M}}
\newcommand{\N}{\mathbb{N}}
\newcommand{\R}{\mathbb{R}}

\renewcommand{\epsilon}{\varepsilon}
\renewcommand{\phi}{\varphi}
\renewcommand{\rho}{\varrho}

\begin{document}

\title[Gasper's determinant theorem, revisited]{Gasper's determinant theorem, revisited}
\author{Markus Sigg}
\address{Freiburg, Germany}
\email{mail@markussigg.de}
\date{April 9, 2018.}

\begin{abstract}
  Let $n \ge 2$ be a natural number, $M$ a real $n \times n$ matrix, $s$ the sum of the entries of $M$ and $q$ the sum of their squares. With $\alpha := s/n$ and $\beta := q/n$, Gasper's determinant bound says that $ |\det M| \le \beta^{n/2}$, and in case of $\alpha^2 \ge \beta$: $$|\det M| \le |\alpha| \left(\frac{n\beta-\alpha^2}{n-1}\right)^{\frac{n-1}2}$$ This article gives a corrected proof of Gasper's theorem and lists some more applications.
\end{abstract}

\maketitle

{
  \small Keywords: determinant bound, Hadamard's inequality, infinite determinant, von Koch matrix.\\
  AMS subjects classification 2010: 15A15, 15A45, 26D07.
}

\section{Introduction}

The present article is primarily a revised version of \cite{gasper}, ironing
out a flaw in the proof of \cite{gasper}, Theorem 1, adding statements about
complex matrices and about infinite determinants and mentioning a few more
applications. We do not repeat the numerical results concerning determinants of
matrices whose entries are a permutation of the numbers $1,\dots,n^2$. See
\cite{wilson} for these.

Throughout, let $n > 1$ be a natural number and $N := \{1,\dots,n\}$. Whenever
not stated otherwise, \emph{matrix} means a real $n \times n$ matrix, the set
of which we denote by $\M$.

For $M \in \M$ and $i,j \in N$ we denote by $M_i$ the $i$-th row of $M$, by
$M^j$ the $j$-th column of $M$, and by $M_{i,j}$ the entry of $M$ at position
$(i,j)$. If $M$ is a matrix or a row or column of a matrix, then by $s(M)$ we
denote the sum of the entries of $M$ and by $q(M)$ the sum of their squares.

The identity matrix is denoted by $I$. By $J$ we name the matrix which has $1$
as all of its entries, while $e$ is the column vector in $\R^n$ with all
entries being $1$. Matrices of the structure $xI+yJ$ will play an important
role, so we state some of their properties:

\begin{lemma}\label{ij}
  Let $x,y \in \R$ and $M := xI + yJ$. Then we have:
  \begin{enumerate}
  \item $\det M = x^{n-1}(x+ny)$
  \item $M$ is invertible if and only if $x \not\in \{0,-ny\}$.
  \item If $M$ is invertible, then $M^{-1} = \frac1x I - \frac y{x(x+ny)}J$.
  \end{enumerate}
\end{lemma}

\begin{proof}
  Because $J = ee^T$, it holds that
  \begin{equation*}
    Me = (xI+yee^T)e = (x+ye^Te)e = (x+ny)e
    \quad\text{and}\quad
    Mv = (xI+yee^T)v = xv
  \end{equation*}
  for all $v \in \R^n$ with $v \perp e$. Hence $M$ has the eigenvalue $x$ with
  multiplicity $n-1$ and the simple eigenvalue $x+ny$. This shows (1). (2)\ is
  an immediate consequence of (1). (3) can be verified by a straight
  calculation.
\end{proof}

\section{Matrices with given entry sum and square sum}

Let $\alpha,\beta \in \R$ with $\beta > 0$. We inspect the following set of
matrices:
\begin{equation*}
  \M_{\alpha,\beta} := \{M \in \M : s(M) = n\alpha,\ q(M) = n\beta\}
\end{equation*}

\begin{lemma}\label{ab}
  \hspace{0pt}\begin{enumerate}
  \item If $\alpha^2 > n\beta$, then $\M_{\alpha,\beta} = \emptyset$.
  \item If $\alpha^2 = n\beta$, then $\det M = 0$ for all $M \in
    \M_{\alpha,\beta}$.
  \item If $\alpha^2 \le n\beta$, then there exists an $M \in
    \M_{\alpha,\beta}$ with
    \begin{equation*}
      \det M = \alpha \left(\frac{n\beta-\alpha^2}{n-1}\right)^{\frac{n-1}2}.
    \end{equation*}
  \item If $\alpha^2 \le \beta$, then there exists an $M \in \M_{\alpha,\beta}$
    with $\det M = \beta^{\frac n2}$.
  \end{enumerate}
\end{lemma}

\begin{proof}
  (1)\ Suppose $\M_{\alpha,\beta} \neq \emptyset$, say $M \in \M_{\alpha,\beta}$.
  Reading $M$ and $J$ as elements of $\R^{n^2}$, Cauchy's inequality gives:
  \begin{equation*}
    \alpha^2
    = \frac1{n^2} \Big(\sum_{i,j=1}^n M_{i,j}\Big)^2
    = \frac1{n^2} \left<M,J\right>^2
    \le \frac1{n^2} {\|M\|}_2^2 \, {\|J\|}_2^2
    = \sum_{i,j=1}^n M_{i,j}^2
    = n\beta
  \end{equation*}
  (2)\ For $\alpha^2 = n\beta$ and $M \in \M_{\alpha,\beta}$, the calculation
  in (1) shows $|\left<M,J\right>| = {\|M\|}_2 \, {\|J\|}_2$. But this holds
  only if $M$ is a scalar multiple of $J$, so we have $\det M = 0$ because of
  $\det J = 0$.

  (3) Suppose $\alpha^2 \le n\beta$. With $\gamma :=
  \left(\frac{n\beta-\alpha^2}{n-1}\right)^\frac12$ set $M := \gamma I +
  \frac1n(\alpha-\gamma)J$. Then $M \in \M_{\alpha,\beta}$, and by Lemma
  \ref{ij}:
  \begin{equation*}
    \det M
    = \gamma^{n-1}\left(\gamma+n\tfrac1n(\alpha-\gamma)\right)
    = \alpha \gamma^{n-1}
  \end{equation*}
  (4) Suppose $\alpha^2 \le \beta$. In case of $\alpha \ge 0$, set $\gamma :=
  \frac12\big(\frac{3\alpha}{\sqrt{\beta}}-1\big)$, so $\gamma^2 \le 1$. Set
  \begin{equation*}
    A := \begin{pmatrix}
      \alpha  & \sqrt{\beta-\alpha^2}\\
      -\sqrt{\beta-\alpha^2} & \alpha
    \end{pmatrix}
    \quad\text{and}\quad
    B := \sqrt{\beta}
    \begin{pmatrix}
      \gamma & \sqrt{1-\gamma^2} & 0\\
      -\sqrt{1-\gamma^2} & \gamma & 0\\
      0 & 0 & 1
    \end{pmatrix}.
  \end{equation*}
  Then $s(A) = 2\alpha$, $q(A) = 2\beta$, $\det A = \beta$, $s(B) = 3\alpha$,
  $q(B) = 3\beta$, $\det B = \beta^\frac32$.  In case of $n = 2k$ with $k \in
  \N$, use $k$ copies of $A$ to build the block matrix
  \begin{equation*}
    M := \begin{pmatrix}
      A \\
      & \ddots\\
      && A
    \end{pmatrix},
  \end{equation*}
  which has the required properties. In case of $n = 2k+1$ with $k \in \N$,
  use $k-1$ copies of $A$ to build the block matrix
  \begin{equation*}
    M :=  \begin{pmatrix}
      A\\
      & \ddots\\
      && A\\
      &&& B
    \end{pmatrix},
  \end{equation*}
  which again satisfies the requirements.

  In case of $\alpha < 0$, an $M' \in \M_{-\alpha,\beta}$ with $\det M' =
  \beta^\frac n2$ exists. For even $n$, $M := -M' \in \M_{\alpha,\beta}$ has
  the requested determinant, while for odd $n$ swapping two rows of $-M'$ gives
  the desired matrix $M$.
\end{proof}

In the proofs of (3) and (4) of Lemma \ref{ab} we have specified matrices
whose determinants will below turn out to be the greatest possible. The
determinant values relate like following:

\begin{lemma}\label{relate}
  For $\alpha^2 \le n\beta$ the inequality
  \begin{equation*}
    |\alpha| \left(\frac{n\beta-\alpha^2}{n-1}\right)^{\frac{n-1}2} \le \beta^\frac n2
  \end{equation*}
  holds, with equality if and only if $\alpha^2 = \beta$.
\end{lemma}

\begin{proof}
  With $f(x) := x \big(\frac{n-x}{n-1}\big)^{n-1}$ for $x \in [0,n]$ we have
  \begin{equation*}
    |\alpha| \left(\frac{n\beta-\alpha^2}{n-1}\right)^{\frac{n-1}2} \beta^{-\frac n2}
    = \sqrt{f\big(\tfrac{\alpha^2}\beta\big)}\,.
  \end{equation*}
  The proof is completed by applying the AM-GM inequality to ${f(x)}^{1/n}$:
  \begin{equation*}
    f{(x)}^\frac1n
    = \textstyle \left(x \left(\frac{n-x}{n-1}\right)^{n-1}\right)^\frac1n
    \le \frac1n \left(x + (n-1) \frac{n-x}{n-1}\right)
    = 1
  \end{equation*}
  with equality if and only if $x = \frac{n-x}{n-1}$, i.\,e.\ if and only if $x
  = 1$.
\end{proof}

\section{Main results}

Let $\alpha,\beta \in \R$ with $\beta > 0$. By Lemma \ref{ab} there exists an
$M \in \M_{\alpha,\beta}$ with $\det M \ne 0$ if and only if $\alpha^2 <
n\beta$. By possibly swapping two rows of $M$, $\det M > 0$ can be achieved. As
$\M_{\alpha,\beta}$ is compact, the determinant function assumes a maximum
value on $\M_{\alpha,\beta}$. Gasper's theorem provides insight into the
properties of the matrices with maximal determinant:

\begin{theorem}[O. Gasper, 2009]\label{theorem}
  Let $\alpha^2 < n\beta$ and $M \in \M_{\alpha,\beta}$ with maximal
  determinant. Then
  \begin{eqnarray*}
    \text{if}\ \alpha^2 \le \beta&:&
    \begin{cases}
      \ (A) & MM^T = \beta I\\[0.25em]
      \ (B) & \det M = \beta^{\frac n2}
    \end{cases}\\
    \text{if}\ \alpha^2 \ge \beta&:&
    \begin{cases}
      \ (C) & s(M_i) = \alpha \quad\text{and}\quad s(M^j) = \alpha \quad
      \text{for all $i,j \in N$}\\[0.25em]
      \ (D) & MM^T = (\beta-\delta) I + \delta J \quad
      \text{with $\delta := \frac{\alpha^2-\beta}{n-1}$, so $\beta-\delta =
        \frac{n\beta-\alpha^2}{n-1}$.}\\[0.25em]
      \ (E) & \det M = |\alpha| \left(\beta-\delta\right)^{\frac{n-1}2}\\[0.25em]
    \end{cases}
  \end{eqnarray*}
\end{theorem}

\begin{proof}
  From Lemma \ref{ab} we know that $\det M > 0$. The matrix $M$ solves an
  extremum problem with equality contraints
  \begin{equation*}
    \text{(P)} \quad\quad
    \begin{cases}
      \ \det X \longrightarrow \max\\
      \ s(X) = n\alpha\\
      \ q(X) = n\beta
    \end{cases}
    (X \in \M^*),
  \end{equation*}
  where $\M^*$ is the set of invertible matrices. The Lagrange function of (P)
  is given by
  \begin{equation*}
  L(X,\lambda,\mu) = \det X - \lambda(s(X)-n\alpha) - \mu(q(X)-n\beta),
  \end{equation*}
  so there exist $\lambda,\mu\in \R$ with $\frac d{dM_{i,j}} L(M,\lambda,\mu) =
  0$ for all $i,j \in N$. By Jacobi's formula
  \begin{equation*}
  \Big(\frac d{dM_{i,j}} \det M\Big)_{i,j} = (\det M) \, {(M^T)}^{-1}
  \end{equation*}
  we get\footnote{This is where a clerical mistake happened in \cite{gasper}.
    Here we have corrected $-\lambda M - 2\mu J$ by $-\lambda J - 2\mu M$ and
    adapted the remainder of the proof accordingly.} $(\det M) \, {(M^T)}^{-1}
  - \lambda J - 2\mu M = 0$, i.\,e.
  \begin{equation}\label{t1}
  (\det M)I = \lambda JM^T + 2\mu MM^T.
  \end{equation}
  Suppose $\mu = 0$. Then applying the determinant function to (\ref{t1}) and
  using $\det J = 0$ would give ${(\det M)}^n = \det(\lambda JM^T) = \det(J)
  \det(\lambda M) = 0$, a contradiction to $\det M > 0$. Hence
  \begin{equation}\label{t2}
    \mu \ne 0.
  \end{equation}
  As $JM^T$ has the diagonal elements $s(M_1),\dots,s(M_n)$, and $MM^T$ has the
  diagonal elements $q(M_1),\dots,q(M_n)$, we get $n \det M = \lambda s(M) +
  2\mu q(M) = \lambda n\alpha + 2\mu n\beta$ by applying the trace function to
  (\ref{t1}), consequently
  \begin{equation}\label{t3}
  \det M = \lambda\alpha + 2\mu\beta.
  \end{equation}
  The symmetry of $(\det M)I$ and the symmetry of $2\mu MM^T$ in (\ref{t1})
  show that $\lambda JM^T$ is symmetric. As all rows of $JM^T$ are identical,
  namely equal to $\left(s(M_1),\dots,s(M_n)\right)$, we obtain
  \begin{equation}\label{t4}
    \lambda s(M_1) = \dots = \lambda s(M_n).
  \end{equation}
  In the following, we inspect the cases $\lambda = 0$ and $\lambda \ne 0$ and prove:
  \begin{equation}\label{t5}
    \begin{cases}
      \ \lambda =   0 \quad\Longrightarrow\quad \alpha^2 \le \beta\ \wedge\ (A)\ \wedge\ (B)\\
      \ \lambda \ne 0 \quad\Longrightarrow\quad \alpha^2 \ge \beta\ \wedge\ (C)\ \wedge\ (D)\ \wedge\ (E)
    \end{cases}
  \end{equation}
  Case $\lambda = 0$: Then (\ref{t3}) reads $\det M = 2\mu\beta$, so taking
  (\ref{t2}) into account and dividing (\ref{t1}) by $2\mu$ gives $\beta I =
  MM^T$, i.\,e.\ (A). From this, (B) follows by applying the determinant
  function. Using the inequality between arithmetic mean and root mean square
  and the fact that the matrix $(1/\sqrt{\beta}) M$ is orthogonal and thus an
  isometry w.r.t.\ the euclidean norm ${\|\ \|}_2$, we get
  \begin{equation}\label{t6}
    \alpha^2
    = \Big(\frac1n \sum_{i=1}^n s(M_i)\Big)^2
    \le \frac1n \sum_{i=1}^n {s(M_i)}^2
    = \frac1n \, {\|Me\|}_2^2
    = \frac1n \, \beta {\|e\|}_2^2
    = \frac1n \, \beta n
    = \beta.
  \end{equation}
  Case $\lambda \ne 0$: Then $s(M_1) = \dots = s(M_n)$ by (\ref{t4}). With
  $s(M_1) + \dots + s(M_n) = s(M) = n\alpha$ this shows $s(M_i) = \alpha$ for
  all $i \in N$. Using $(\det M)I = \lambda M^TJ + 2\mu M^TM$ instead of (1)
  yields $s(M^j) = \alpha$ for all $j \in N$, so (C) is done. Furthermore,
  $JM^T = \alpha J$, and (\ref{t1}) becomes
  \begin{equation}\label{t7}
  2\mu MM^T = (\det M)I - \lambda\alpha J,
  \end{equation}
  hence $q(M_i) = (MM^T)_{i,i} = (\det M - \lambda\alpha)/(2\mu)$ for all
  $i \in N$, so $q(M_1) = \dots = q(M_n)$. Using $q(M_1) + \dots + q(M_n) =
  q(M) = n\beta$ shows
  \begin{equation}\label{t8}
  {(MM^T)}_{i,i} = q(M_i) = \beta \quad \text{for all $i \in N$.}
  \end{equation}
  Let $i,j \in N$ with $i \ne j$. Then (\ref{t7}) gives ${(MM^T)}_{i,k} =
  -\lambda\alpha/(2\mu)$ for all $k \in N \setminus \{i\}$. With
  \begin{equation*}
    \sum_{k=1}^n {(MM^T)}_{i,k}
    = \sum_{k=1}^n \sum_{p=1}^n M_{i,p} \, M_{k,p}
    = \sum_{p=1}^n M_{i,p} \, s(M^p)\\
    = \sum_{p=1}^n M_{i,p} \, \alpha
    = s(M_i) \, \alpha
    = \alpha^2
  \end{equation*}
  and (\ref{t8}) we get
  \begin{equation*}
    {(MM^T)}_{i,j}
    = \frac1{n-1} \sum_{k \ne i} {(MM^T)}_{i,k}
    = \frac1{n-1} \left( \sum_{k=1}^n {(MM^T)}_{i,k} - {(MM^T)}_{i,i} \right)
    = \frac{\alpha^2 - \beta}{n-1} = \delta,
  \end{equation*}
  which, again with (\ref{t8}), proves (D). With Lemma \ref{ij} this
  yields
  \begin{equation*}
    {(\det M)}^2
    = \det(MM^T)
    = (\beta-\delta)^{n-1} (\beta-\delta+n\delta)
    = (\beta-\delta)^{n-1} \alpha^2 ,
  \end{equation*}
  and taking the square root gives (E). Suppose $\alpha^2 < \beta$. Then by
  Lemma \ref{ab} there would exist an $M' \in \M_{\alpha,\beta}$ with $\det M'
  = \beta^\frac n2$, and by Lemma \ref{relate}
  \begin{equation*}
    \det M
    = |\alpha| {(\beta-\delta)}^{\frac{n-1}2}
    < \beta^{\frac n2}
    = \det M',
  \end{equation*}
  which contradicts the maximality of $\det M$. Hence $\alpha^2 \ge \beta$.

  We have now proved (\ref{t5}) and are ready to deduce the statements of the
  theorem: If $\alpha^2 < \beta$, then (\ref{t5}) shows that $\lambda = 0$ and
  thus (A) and (B). If $\alpha^2 > \beta$, then (\ref{t5}) shows that $\lambda
  \ne 0$ and thus (C), (D) and (E). Finally suppose $\alpha^2 = \beta$. Then
  $\delta = 0$, hence (A) $\Longleftrightarrow$ (D) and (B)
  $\Longleftrightarrow$ (E). If $\lambda \ne 0$, then (\ref{t5}) shows (C), (D)
  and (E), from which (A) and (B) follow. If $\lambda = 0$, then (\ref{t5})
  shows (A) and (B), from which (D) and (E) follow. It remains to prove (C) in
  the case of $\alpha^2 = \beta$ and $\lambda = 0$. To this purpose, look at
  (\ref{t6}) again, where $\alpha^2 = \beta$ shows that $s(M_1) = \dots =
  s(M_n)$, and (C) follows as in the case $\lambda \ne 0$.
\end{proof}

For calculating upper bounds for the determinants of given matrices, we note
this handy consequence of Theorem \ref{theorem}:

\begin{proposition}\label{real}
  Let $M \in \M$, $\alpha := \frac1n s(M)$, $\beta := \frac1n q(M)$, $\kappa :=
  \frac{n\beta-\alpha^2}{n-1}$. Then
  \begin{eqnarray*}
    \text{if}\ \alpha^2 < \beta &:& |\det M| \le \beta^{\frac n2}\\
    \text{if}\ \alpha^2 = \beta &:& |\det M| \le |\alpha| \kappa^{\frac{n-1}2} = \beta^{\frac n2}\\
    \text{if}\ \alpha^2 > \beta &:& |\det M| \le |\alpha| \kappa^{\frac{n-1}2} < \beta^{\frac n2}
  \end{eqnarray*}
\end{proposition}

\begin{proof}
  This is trivial if $\det M = 0$. In case of $\det M \ne 0$ we get $\alpha^2 <
  n\beta$ by Lemma \ref{ab}, and the stated inequalities are true by Lemma
  \ref{relate} and Theorem \ref{theorem}.
\end{proof}

Note that Lemma \ref{relate} says that $|\alpha|\kappa^{\frac{n-1}2} <
\beta^{\frac n2}$ is true in case of $\alpha^2 < \beta$, too. However, as the
following examples demonstrates, $|\det M|$ is not necessarily bounded by the
left hand side in this situation:
\begin{equation*}
    M :=  \begin{pmatrix}
      1 & 0\\
      0 & -1
    \end{pmatrix}\ ,\
    |\det M| = 1 \ ,\
    |\alpha|\kappa^{\frac{n-1}2} = 0
\end{equation*}
Proposition \ref{real} can be used to derive bounds for the determinants of
complex matrices also:

\begin{corollary}\label{complex}
  Let $A,B \in \M$, $M := A + iB$, $\alpha := \frac1n s(A)$, $\beta := \frac1n
  (q(A)+q(B))$, $\kappa := \frac{2n\beta-\alpha^2}{2n-1}$. Then
  \begin{eqnarray*}
    \text{if}\ \alpha^2 < \beta &:& |\det M| \le \beta^{\frac n2}\\
    \text{if}\ \alpha^2 = \beta &:& |\det M| \le {|\alpha|}^\frac12 \kappa^\frac{2n-1}4 = \beta^{\frac n2}\\
    \text{if}\ \alpha^2 > \beta &:& |\det M| \le {|\alpha|}^\frac12 \kappa^\frac{2n-1}4 < \beta^{\frac n2}
  \end{eqnarray*}
\end{corollary}

\begin{proof}
  For the real $2n \times 2n$-matrix
  \begin{equation*}
    M' := \begin{pmatrix}
      A & B\\
      -B & A
    \end{pmatrix}
  \end{equation*}
  we have $s(M') = 2s(A)$ and $q(M') = 2q(A)+2q(B))$, hence $\alpha' :=
  \frac1{2n} s(M') = \frac1n s(A) = \alpha$ and $\beta' := \frac1{2n} q(M') =
  \frac1n (q(A) + q(B)) = \beta$, and Proposition \ref{real} applied to $M'$
  gives
  \begin{eqnarray*}
    \text{if}\ \alpha^2 < \beta &:& |\det M'| \le \beta^{\frac{2n}2}\\
    \text{if}\ \alpha^2 = \beta &:& |\det M'| \le |\alpha| \kappa^{\frac{2n-1}2} = \beta^{\frac{2n}2}\\
    \text{if}\ \alpha^2 > \beta &:& |\det M'| \le |\alpha| \kappa^{\frac{2n-1}2} < \beta^{\frac{2n}2}
  \end{eqnarray*}
  The claimed inequalities follow by using $\det M' = {|\det M|}^2$, see
  \cite{bernstein}, Fact 3.24.7 vii).
\end{proof}

To get the more attractive case of $\alpha^2 > \beta$ in Corollary
\ref{complex}, it can help to recall the equality $|\det(A+iB)| = |\det(B+iA)|$
and apply Corollary \ref{complex} to the latter matrix. As an example take $A
:= \left(\begin{smallmatrix} 0 & 0\\ 0 & 0 \end{smallmatrix}\right)$ and $B :=
\left(\begin{smallmatrix} 1 & 1\\ 1 & 1 \end{smallmatrix}\right)$. For $A+iB$
we get $\alpha = 0$, $\beta = 2$ and the bound $|\det(A+iB)| \le 2$. But $B+iA$
gives $\alpha = 2$, $\beta = 2$ and the bound $|\det(B+iA)| \le 4 \cdot
27^{-\frac14} \approx 1.75$.

For a real matrix $M$, i.\,e. $B = 0$, Corollary \ref{complex} can yield a
larger bound than Proposition \ref{real}, and it can give different bounds for
$M$ and for $iM$. For example the matrix $M := \left(\begin{smallmatrix} 1 &
  1\\ 0 & 1 \end{smallmatrix}\right)$ gets the bound $\frac34\sqrt3 \approx
1.30$ from Proposition \ref{real} and the bound $\frac14 125^\frac14 \sqrt{3}
\approx 1.45$ from Corollary \ref{complex}, while Corollary \ref{complex}
applied to $iM$ gives the bound $1.5$. This unhappy situation prompts for
\begin{question}
    Is there a better way to transfer Proposition \ref{real} to complex
    matrices?
\end{question}

\section{Applications}

The bounds for $\det M$ in Proposition \ref{real} and Corollary \ref{complex}
refer only to $s(M)$ and $q(M)$ and so do not use any positional information.
As sections \ref{sectionhadamard}--\ref{sectionkoch} show, they can still serve
for deducing interesting inequalities. But the strength of Proposition
\ref{real} manifests better when it is applied to problems like in sections
\ref{sectionarith1} and \ref{sectionarith2}.

\subsection{Hadamard's inequality}\label{sectionhadamard}

For a complex $n \times n$ matrix $M$ with $\left|M_{i,j}\right| \le 1$ for all
$i,j \in N$, Corollary \ref{complex} shows:
\begin{equation*}
  |\det M|
  \le \beta^{\frac n2}
  = \left(\frac1n \sum_{i,j=1}^n {|M_{i,j}|}^2\right)^{\frac n2}
  \le \left(\frac1n \sum_{i,j=1}^n 1\right)^{\frac n2}
  = n^{\frac n2}
\end{equation*}
This is Hadamard's inequality, see \cite{hadamard}. For $\gamma \ge 0$ and
$\left|M_{i,j}\right| \le \gamma$ for all $i,j \in N$ we get in the same way
the inequality $|\det M| \le \gamma^n\,n^{n/2}$. However, Hadamard's more
general bound
\begin{equation*}
  |\det M|
  \le \prod_{i=1}^n \left(\sum_{j=1}^n \left|M_{i,j}\right|^2\right)^\frac12
\end{equation*}
cannot be derived from the $\beta^\frac n2$ bound. The AM-GM inequality shows
that this bound is less than or equal to the $\beta^\frac n2$ bound, and there
are cases where it is strictly smaller. But matrices exist where the case
$\alpha^2 > \beta$ in Proposition \ref{real} applies and yields a bound that is
better than Hadamard's. As an example, Hadamard's bound for the matrix
$\left(\begin{smallmatrix} 1 & 2\\ 2 & 3\end{smallmatrix}\right)$ is
  $\sqrt{65}$ while Proposition \ref{real} gives the bound $\sqrt{32}$.

\subsection{Best's inequality}

If $M_{i,j} \in \{-1,1\}$ for all $i,j \in N$ and $|\det M| = n^{n/2}$,
i.\,e.\ $M$ is a \emph{Hadamard matrix}, then Proposition \ref{real} shows that
$\alpha^2 \le \beta$ must hold. The value $s(M)$ is called the \emph{excess} of
$M$. Because $q(M) = n^2$ in case of $M_{i,j} \in \{-1,1\}$, Proposition
\ref{real}\ yields an upper bound for the excess:
\begin{equation*}
  M\ \text{is a Hadamard matrix}
  \quad\Longrightarrow\quad
  s(M) \le n\sqrt{n}
\end{equation*}
This is known as Best's inequality, see \cite{best}.

\subsection{Inequality of determinant and trace}

For a positive integer $m$ and positive definite matrices $A$ and $B$, Theorem
2.8 in \cite{dannan} says
\begin{equation*}
  (\det AB)^\frac mn \le \frac1n \tr(A^m B^m),
\end{equation*}
which for $m = 1$ and $B = A^T$ reads $(\det A)^\frac2n \le \frac1n \tr(AA^T)$.
As $\tr(AA^T) = q(A)$, Proposition \ref{real} shows that, for the latter
inequality, $A$ does not need to be positive definite.

\subsection{Ryser's inequality}

If $M_{i,j} \in \{0,1\}$ for all $i,j \in N$ and $t := s(M)$ is the number of
1's in $M$, then in Proposition \ref{real} with $k := t/n$ we have $\alpha =
\beta = k$ and get
\begin{eqnarray*}
  \text{if}\ t < n &:& |\det M| \le k^\frac n2\\[0.75em]
  \text{if}\ t = n &:& |\det M| \le 1\\
  \text{if}\ t > n &:& |\det M| \le k^\frac{n+1}2 \left(\frac{n-k}{n-1}\right)^\frac{n-1}2
\end{eqnarray*}
The inequality for the case $t > n$ is Ryser's determinant bound \cite{ryser},
Theorem 3. For the case of $t = 2n$ this was improved by Bruhn and Rautenbach,
see \cite{bruhn}, Theorem 3, and \cite{pfoertner1}.

\subsection{The inequalities of Brent, Osborne and Smith}

Let $\epsilon > 0$, $E \in \M$ with $|E_{i,j}| \le \epsilon$ for all $i,j \in
N$ and $M := I-E$. Then
\begin{equation*}
  \beta
  \le \frac1n \left( n(1+\epsilon)^2 + (n^2-n) \epsilon^2 \right)
  = 1 + 2\epsilon + n \epsilon^2
\end{equation*}
in Proposition \ref{real} gives
\begin{equation*}
  |\det M| \le \left( 1 + 2\epsilon + n \epsilon^2 \right)^\frac n2.
\end{equation*}
This is \cite{brent}, Theorem 3\,(8). In case of $E_{i,i} = 0$ for all $i \in
N$ we have
\begin{equation*}
  \beta
  \le \frac1n \left( n + (n^2-n) \epsilon^2 \right)
  = 1 + (n-1) \epsilon^2
\end{equation*}
and so
\begin{equation*}
  |\det M| \le \left( 1 + (n-1) \epsilon^2 \right)^\frac n2,
\end{equation*}
which is \cite{brent}, Theorem 3\,(9). While, as is demonstrated in
\cite{brent}, both inequalities follow from Hadamard's inequality, the above
reasoning shows that these bounds do not depend on the arrangement of the
dominant entries, and it can also be used to derive bounds if the number of
dominant entries is different from $n$.

\subsection{Determinants of von Koch matrices}\label{sectionkoch}

Let $A := \left(A_{i,j} : i,j = 1 \dots \infty\right)$ be an infinite
matrix of real numbers $A_{i,j}$ with
\begin{equation*}
  \sum_{i=1}^\infty \left|A_{i,i}\right| < \infty
  \quad\text{and}\quad
  \sum_{i,j=1}^\infty \left|A_{i,j}\right|^2 < \infty
\end{equation*}
and $A(n) := \left(A_{i,j} : i,j = 1 \dots n\right)$ for $n \in \N$ be the $n
\times n$ finite submatrix. Then by \cite{koch} or \cite{gohberg}, page 170,
the \textit{infinite determinant}
\begin{equation*}
  \det(I - A) := \lim_{n \to \infty} \det(I - A(n)),
\end{equation*}
is well-defined, where $I$ is the suitable finite or infinite identity matrix.
By Proposition \ref{real}
\begin{equation*}
  |\det(I - A(n))|
  \le \left(\frac1n\sum_{i,j=1}^n \left(\delta_{i,j} - A_{i,j}\right)^2\right)^\frac n2
  = \left(1 + \frac1n\sum_{i,j=1}^n A_{i,j}^2 - \frac2n\sum_{i=1}^n A_{i,i}\right)^\frac n2,
\end{equation*}
 and as $\left(1+\frac{x_n}n\right)^\frac n2 \to e^\frac x2$ for $n \to \infty$
 for a real sequence $(x_n)$ with limit $x$:
\begin{equation}\label{infbound}
  |\det(I - A)|
  \le \exp\left(\frac12\sum_{i,j=1}^\infty A_{i,j}^2 - \sum_{i=1}^\infty A_{i,i}\right)
\end{equation}
One might want to apply the better bound of Proposition \ref{real}'s case
$\alpha^2 > \beta$, but an evaluation reveals that for $n \to \infty$ this
gives the identical inequality (\ref{infbound}).

\subsection{Matrices whose entries are a permutation of an arithmetic progression}\label{sectionarith1}

\begin{proposition}\label{arith1}
  Let $p,q$ be real numbers with $q > 0$ and $M$ a matrix whose entries are a
  permutation of the numbers $p, p+q, \dots, p+(n^2-1)q$. Set
  \begin{equation*}
    r := \frac pq + \frac{n^2-1}2
    \quad,\quad
    \rho := \frac{n^3+n^2+n+1}{12}
    \quad\text{and}\quad
    \sigma := n q^2 \left(r^2 + \frac{n^4-1}{12}\right).
  \end{equation*}
  Then
  \begin{eqnarray*}
    \text{if}\ r^2 < \rho &:& |\det M| \le \sigma^\frac n2\\
    \text{if}\ r^2 = \rho &:& |\det M| \le n^n q^n \, |r| \rho^\frac{n-1}2 = \sigma^\frac n2\\
    \text{if}\ r^2 > \rho &:& |\det M| \le n^n q^n \, |r| \rho^\frac{n-1}2 < \sigma^\frac n2
  \end{eqnarray*}
\end{proposition}

\begin{proof}
  With $\alpha$ and  $\beta$ as in Proposition \ref{real} a calculation
  shows $\alpha^2-\beta = n(n-1)q^2(r^2-\rho)$, hence $\sgn(\alpha^2-\beta)
  = \sgn(r^2-\rho)$, and the bounds noted in Proposition \ref{real} yield the
  asserted inequalities for $|\det M|$.
\end{proof}

\begin{corollary}\label{corarith1}
  If $M$ is a matrix whose entries are a permutation of $0,\dots,n^2-1$, then
  \begin{equation*}
    |\det M| \le n^n\,\frac{n^2-1}2 \left(\frac{n^3+n^2+n+1}{12}\right)^\frac{n-1}2.
  \end{equation*}
  If $M$ is a matrix whose entries are a permutation of $1,\dots,n^2$, then
  \begin{equation*}
    |\det M| \le n^n\,\frac{n^2+1}2 \left(\frac{n^3+n^2+n+1}{12}\right)^\frac{n-1}2.
  \end{equation*}
\end{corollary}

\begin{proof}
  Apply Proposition \ref{arith1} to $(p,q) := (0,1)$ and to $(p,q) := (1,1)$,
  respectively. In both applications it is easy to see that $r^2 > \rho$, which
  yields the stated bound.
\end{proof}

Actual maximal determinants for this kind of matrices suggest

\begin{question}
  Let $b(n)$ be the upper bound given in Corollary \ref{corarith1} for matrices
  with entries $1,\dots,n^2$. Regarding \cite{wilson}, do we have
  \begin{equation*}
    \lim_{n \to \infty} \frac{A085000(n)}{b(n)} = 1 ?
    \end{equation*}
\end{question}

\subsection{A variation of the previous theme}\label{sectionarith2}

\begin{proposition}\label{arith2}
  Let $p,q$ be real numbers with $q > 0$ and $M$ a matrix such that each of the
  numbers $p, p+q, \dots, p+(n-1)q$ appears $n$ times in $M$. Set
  \begin{equation*}
    r := \frac pq + \frac{n-1}2
    \quad,\quad
    \rho := \frac{n+1}{12}
    \quad\text{and}\quad
    \sigma := n q^2 \left(r^2 + \frac{n^2-1}{12}\right).
  \end{equation*}
  Then
  \begin{eqnarray*}
    \text{if}\ r^2 < \rho &:& |\det M| \le \sigma^\frac n2\\
    \text{if}\ r^2 = \rho &:& |\det M| \le n^n q^n |r| \rho^\frac{n-1}2 = \sigma^\frac n2\\
    \text{if}\ r^2 > \rho &:& |\det M| \le n^n q^n |r| \rho^\frac{n-1}2 < \sigma^\frac n2
  \end{eqnarray*}
\end{proposition}

\begin{proof}
  We have again $\alpha^2-\beta = n(n-1)q^2(r^2 -\rho)$ and can apply
  Proposition \ref{real}.
\end{proof}

\begin{corollary}\label{corarith2}
  If $M$ is a matrix such that each of the numbers $0,\dots,n-1$ appears
  $n$ times in M, then
  \begin{equation*}
    |\det M| \le n^n\,\frac{n-1}2 \left(\frac{n+1}{12}\right)^\frac{n-1}2.
  \end{equation*}
  If $M$ is a matrix such that each of the numbers $1,\dots,n$ appears
  $n$ times in M, then
  \begin{equation*}
    |\det M| \le n^n\,\frac{n+1}2 \left(\frac{n+1}{12}\right)^\frac{n-1}2.
  \end{equation*}
\end{corollary}

\begin{proof}
  Apply Proposition \ref{arith2} to $(p,q) := (0,1)$ and to $(p,q) := (1,1)$.
\end{proof}

Again we would like to pose the

\begin{question}
  Let $b(n)$ be the upper bound given in Corollary \ref{corarith2} for matrices
  with entries $1,\dots,n$. Regarding \cite{pfoertner2}, do we have
  \begin{equation*}
    \lim_{n \to \infty} \frac{A301371(n)}{b(n)} = 1 ?
    \end{equation*}
\end{question}

\end{document}